\documentclass[11pt]{article}
\usepackage{iwona}
\usepackage{amsmath,amsthm,amsfonts,amssymb,enumerate,calc}
\usepackage[colorlinks=true,citecolor=black,linkcolor=black,urlcolor=blue]{hyperref}
\usepackage[numbers,sort&compress]{natbib}
\usepackage[capitalise]{cleveref}

\usepackage[lmargin=35mm,rmargin=35mm,tmargin=30mm,bmargin=30mm]{geometry}
\renewcommand{\baselinestretch}{1.1}
\renewcommand{\thefootnote}{\fnsymbol{footnote}}	
\allowdisplaybreaks

\newcommand{\msn}[1]{MR:\,\href{http://www.ams.org/mathscinet-getitem?mr=MR#1}{#1}}
\newcommand{\MSN}[2]{MR:\,\href{http://www.ams.org/mathscinet-getitem?mr=MR#1}{#1}}
\newcommand{\doi}[1]{doi:\,\href{http://dx.doi.org/#1}{#1}}
\newcommand{\cs}[1]{CiteSeer:\,\href{http://citeseerx.ist.psu.edu/viewdoc/summary?doi=#1}{#1}}

\newcommand{\PP}[1]{\textup{(P{#1})}}
\newcommand{\QQ}[1]{\textup{(Q{#1})}}
\newcommand{\dis}{\mu}
\newcommand{\bad}{\nu}

\theoremstyle{plain}
\newtheorem{thm}{Theorem}
\newtheorem{lem}[thm]{Lemma}

\newtheorem{prop}[thm]{Proposition}

\newtheorem{claim}{Claim}
\theoremstyle{definition}

\renewcommand{\geq}{\geqslant}
\renewcommand{\leq}{\leqslant}
\newcommand{\ceil}[1]{\lceil{#1}\rceil}

\newcommand{\floor}[1]{\lfloor{#1}\rfloor}

\newcommand{\half}{\ensuremath{\protect\tfrac{1}{2}}}
\newcommand{\brac}[1]{\left(#1\right)}

\begin{document}

\title{\bf Forcing a sparse minor}

\author{Bruce Reed\footnotemark[2] \qquad David~R.~Wood\footnotemark[3]}

\date{3 February 2014, revised 7 May 2015}

\maketitle

\begin{abstract}
This paper addresses the following question for a given graph $H$: what is the minimum number $f(H)$ such that every graph with average degree at least $f(H)$ contains $H$ as a minor? Due to connections with Hadwiger's Conjecture, this question has been studied in depth when $H$ is a complete graph. Kostochka and Thomason independently proved that $f(K_t)=ct\sqrt{\ln t}$. More generally, Myers and Thomason determined $f(H)$ when $H$ has a super-linear number of edges. We focus on the case when $H$ has a linear number of edges. Our main result, which complements the result of Myers and Thomason,  states that if $H$ has $t$ vertices and  average degree $d$ at least some absolute constant, then $f(H)\leq 3.895\sqrt{\ln d}\,t$. Furthermore, motivated by the case when $H$ has small average degree, we prove that if $H$ has $t$ vertices and $q$ edges, then $f(H) \leq t+6.291q$ (where the coefficient of 1 in the $t$ term is best possible). 

\end{abstract}


\footnotetext[2]{Canada Research Chair in Graph Theory, School of Computer Science, McGill University, Montreal, Canada (\texttt{breed@cs.mcgill.ca}). National Institute of Informatics, Japan. Funded by NSERC and the  ERATO Kawarabayashi Large Graph Project.}

\footnotetext[3]{School of Mathematical Sciences, Monash University, Melbourne, Australia
  (\texttt{david.wood@monash.edu}). Research supported by  the Australian Research Council.}

\renewcommand{\thefootnote}{\arabic{footnote}}


\section{Introduction}
\label{sec:Intro}

A graph $H$ is a \emph{minor} of a graph $G$ if a graph isomorphic to $H$ can be obtained from a subgraph of $G$ by contracting edges. This paper studies average degree conditions that force an $H$-minor. In particular, it focuses on the infimum of all numbers $d$ such that every graph with average degree at least $d$ contains $H$ as a minor, which we denote by $f(H)$. We are interested in determining bounds on $f(H)$ that are a function of the number of edges and vertices of $H$. 

We distinguish two types of graphs $H$ (or to be more precise, families of graphs $H$). We consider $H$ to be `dense' if $|E(H)|\geq|V(H)|^{1+\tau}$ for some constant $\tau>0$. On the other hand, we consider $H$ to be `sparse' if $|E(H)|\leq c|V(H)|$ for some constant $c$ (independent of $|V(H)|$). This paper focuses on $f(H)$ for graphs $H$ that are not dense, and especially those that are sparse. 

Previous work in this field concerns dense $H$. Indeed, largely motivated by Hadwiger's Conjecture, $f(H)$ was first studied for $H=K_t$, the complete graph on $t$ vertices. \citet{Dirac64} proved that for $t\leq 5$, every $n$-vertex $K_t$-minor-free graph has at most $(t-2)n-\binom{t-1}{2}$ edges, and this bound is tight.  \citet{Mader68} extended this result for $t\leq 7$. It follows that $f(K_t)=2t-4$ for $t\leq 7$. For $t\geq 8$ there are $K_t$-minor-free graphs with more than $(t-2)n-\binom{t-1}{2}$ edges. However, results of \citet{Jorg94} and \citet{ST06} respectively imply that $f(K_8)=12$ and $f(K_9)=14$. Thus $f(K_t)=2t-4$ for $t\leq 9$. \citet{Song04} proved that $f(K_{10})\leq 22$ and  $f(K_{11})\leq 26$, and conjectured that both these bounds can be improved. 

The first upper bound on $f(K_t)$ for general $t$ was due to \citet{Mader67}, who proved that $f(K_t)\leq 2^{t-2}$. \citet{Mader68}  later proved that $f(K_t)\in \mathcal{O}(t\ln t)$. \citet{Kostochka82,Kostochka84} and \citet{delaVega} (based on the work of \citet{BCE80}) independently proved the lower bound, $f(K_t)\in\Omega(t\sqrt{\ln t})$. A matching upper bound of $f(K_t)\in\mathcal{O}(t\sqrt{\ln t})$ was independently proved by \citet{Kostochka82,Kostochka84} and \citet{Thomason84}. Later, \citet{Thomason01} determined the asymptotic constant:
\begin{equation}
\label{CompleteGraph}
f(K_t)=(\alpha+o(1))t\sqrt{\ln t},
\end{equation}
where $\alpha=0.638\ldots$ is an explicit constant, and $o(1)$ denotes a term tending to $0$ as $t\rightarrow \infty$. \citet{Myers-CPC02}  characterised the extremal $K_t$-minor-free graphs as unions of pseudo-random graphs. 


\citet{MT-Comb05} generalised \eqref{CompleteGraph} for dense graphs $H$ as follows. They introduced a graph parameter $\gamma$ with the property that if $t=|V(H)|$ then 
\begin{equation}
\label{DenseGraph}
f(H)=(\alpha\,\gamma(H)+o(1))t\sqrt{\ln t},
\end{equation}
where $\gamma(H)\leq1$ and $o(1)$ denotes a term (slowly) tending to $0$ as $t\rightarrow\infty$. Note that when $H$ is sparse, the $o(1)$ term might dominate $\gamma(H)$, in which case this result says little about $f(H)$, as discussed by \citet[Section~7]{MT-Comb05}. For example, 
\eqref{DenseGraph} does not determine $f$ for specially structured graphs such as unbalanced complete bipartite graphs; see \cref{Special} below.

Moreover, \citet{MT-Comb05} proved that if $H$ has $t^{1+\tau}$ edges, for some constant $\tau>0$, then $\gamma(H) \leq \sqrt{\tau}$ (with equality for almost all $H$ and for all regular $H$). That is, if $H$ has average degree $d=2t^\tau$, then 
\begin{equation}
\label{DenseGraphUpperBound}
f(H)\leq \alpha \sqrt{\ln d}\,t + o(t\sqrt{\ln t}).
\end{equation}
Since the $o(1)$ term tends to $0$ slowly, this bound also says little when $H$ is sparse. 

\subsection{Non-Dense Graphs $H$}

With respect to typical non-dense graphs, we prove the following theorem in the same direction as \eqref{DenseGraphUpperBound}  except it does apply when $G$ is not dense. 

\begin{thm}
\label{General}
There is an absolute constant $d_0$, such that for every graph $H$ with $t$ vertices and average degree $d\geq d_0$, 
$$f(H)\leq 3.895\sqrt{\ln d}\,t.$$
\end{thm}


%

The lower bounds on $f(H)$ due to  \citet{MT-Comb05}  apply even when $H$ is sparse. It follows that \cref{General} is tight up to a constant factor for numerous graphs $H$. In particular, if $H$ is sufficiently large, and is random, or regular, or even if the maximum and minimum degrees are close, then \cref{General} is tight\footnote{For example, consider a $d$-regular graph $H$ on $t$ vertices. In the notation of \citet{MT-Comb05}, $\tau(H)=\log_t(\frac{d}{2})$. Their Theorem 4.8 and Corollary 4.9 imply that $\gamma(H)\rightarrow\sqrt{\tau(H)}$  as $t\rightarrow\infty$. Let $n := \floor{\gamma(H)t\sqrt{\ln t}}\rightarrow \sqrt{\ln d}t$. \citet[Theorem 2.3]{MT-Comb05} prove that $H$ is a minor of a random graph $G(n,\half)$ with probability tending to $0$ as $t\rightarrow\infty$. Thus some graph with average degree $\frac{n}{2}$ contains no $H$-minor. Thus $f(H)\geq c\sqrt{\ln d}t$.}. \cref{General} is proved in Section~\ref{sec:General}. 

When $d$ is very small, \cref{General} is not applicable. Thus, motivated by the case of graphs $H$ with small average degree, we investigate linear bounds of the form $$f(H)\leq\alpha|V(H)|+\beta|E(H)|$$ for explicit constants $\alpha$ and $\beta$. 
 A first question in this regard is the smallest possible values for $\alpha$ and $\beta$. We can push $\beta$ as close to $0$ as we like. Indeed, \cref{General} immediately implies that  for every $\beta>0$ there is a constant $c=c(\beta)$ such that  $f(H)\leq ct+\beta q$ 
for every graph $H$ with $t$ vertices and $q$ edges,
On the other hand, $\alpha\geq1$ in any such bound, since $K_{t-1}$ has average degree $t-2$ but does not contain the graph with $t$ vertices and no edges as a minor. At this extremity we prove the following (in Section~\ref{sec:Linear}): 
  
\begin{thm}
\label{pMain}
For every  graph $H$ with $t$ vertices and $q$ edges,
$$f(H) \leq t+6.291q.$$
\end{thm}

Note that $\beta\geq\frac{1}{3}$ in any bound of the form $f(H)\leq t+\beta q$ (since in Section~\ref{OpenProblems} we observe that if $H$ consists of $k\geq1$ disjoint triangles, then 
$f(H)=4k-2=t+\frac{q}{3}-2$). 

%


Also, note that a linear bound of the form $f(H)\leq \alpha t+ \beta q$ can also be concluded from a theorem of \citet[Theorem~5.1]{FoxSud09} in conjunction with an old lemma of Mader (our \cref{EasyMader}).

\subsection{Specially Structured Graphs}
\label{Special}

Attention in the literature has also been focused on specially structured graphs, as we now discuss. 
We propose some open problems in this regard in our concluding section.

Let $K_{s,t}$ be the complete bipartite graph with $s\leq t$. First consider when $s$ is small. \citet{CRS} proved that the maximum number of edges in a $K_{2,t}$-minor-free graph is at most $\frac{1}{2}(t+1)(n-1)$, which is tight for infinitely many values of $n$. This implies that $f(K_{2,t})=t+1$. \citet{Myers03} had earlier proved the same result for sufficiently large $t$. \citet{KP10} proved that for $t\geq 6300$ and $n\geq t+3$, every $n$-vertex graph $G$ with more than $\half(t+3)(n-2)+1$ edges has a $K_{3,t}$ minor, and this bound is tight. Thus $f(K_{3,t})=2t+6$ for  $t\geq 6300$. 

Now consider general complete bipartite graphs. \citet{Myers03} conjectured that for every integer $s$ there exists a positive constant $c$ such that $f(K_{s,t})\leq ct$ for every integer $t$. \citet{KO05}  and \citet{KP08} independently proved certain strengthenings of this conjecture. Let $K^*_{s,t}$ denote the graph obtained from $K_{s,t}$ by adding all edges between the $s$ vertices of degree $t$. Of course, $f(K_{s,t})\leq f(K^*_{s,t})$. \citet{KO05} proved that $f(K^*_{s,t})\leq (1+\epsilon)t$ for all $t\geq t(\epsilon)$ and $s\leq \epsilon^6t/\log t$. \citet{KP08} proved that $t+3s-5\sqrt{s}\leq f(K_{s,t})\leq f(K^*_{s,t})\leq t+3s$ for  $t > (180s \log_2 s)^{1+6s \log_2 s}$. \citet{KP12} refined their method  to conclude a similar upper bound of $f(K^*_{s,t})\leq t+8s\log_2s$ under a more reasonable assumption about $t$, namely that $t/\log_2t\geq 1000s$. This result is best possible in the sense that the $1000$ and $8$ cannot be simultaneously  reduced to $1/18$, say. Again, considering $K_{s,t}$ rather than $K^*_{s,t}$ does not significantly affect the bounds. 

See \citep{Verstraete,Thom07,Thom08} for various results concerning average or minimum degree conditions that force several copies of a given graph as a minor or subdivision.

\section{A Minor with Large Minimum Degree}

The standard approach to find an $H$-minor in a graph $G$ of high average degree involves first finding a minor $G'$ of $G$ with high minimum degree and few vertices. Then it is shown that $H$ is a minor of $G'$ and hence of $G$. This approach was introduced by \citet{Mader68}. This section presents a variety of results proving that such $G'$ exist. Many of these results can be found in reference \citep{Mader68}, but since this paper is written in German, we include all the proofs. A graph $G$ is \emph{minor-minimal} with respect to some set $\mathcal{G}$ of graphs if $G\in\mathcal{G}$ and every proper minor of $G$ is not in $\mathcal{G}$. 

\begin{lem}
\label{BasicMinorMinimal} 
Let $G$ be a minor-minimal graph with average degree at least $d$. Then every edge of $G$ is in at least $\floor{\frac{d}{2}}$ triangles, and every vertex has degree at least $\floor{\frac{d}{2}}+1$. 
\end{lem}

\begin{proof}
Say $G$ has $n$ vertices and $m$ edges, where $2m\geq dn$. 
Suppose on the contrary that $G$ contains an edge $e$ in $t<\floor{\frac{d}{2}}$ triangles. 
Then $G/e$ has $n-1$ vertices and $m-1-t$ edges. 
Thus $G/e$ has average degree $\frac{2(m-1-t)}{n-1}>\frac{2m}{n}\geq d$.
Hence $G$ is not minor-minimal with average degree at least $d$. 
This contradiction proves that every edge of $G$ is in at least $\floor{\frac{d}{2}}$ triangles. 
Thus every vertex has degree at least $\floor{\frac{d}{2}}+1$. 
\end{proof}

\begin{lem}
\label{EasyMader} 
Every graph with average degree $d\geq 1$ contains a minor with at most $\ceil{\frac{d^2+1}{d+1}}$ vertices and minimum degree at least $\floor{\frac{d}{2}}$, as well as a minor with at most $\ceil{\frac{d^2+1}{d+1}}+1$ vertices and minimum degree at least $\floor{\frac{d}{2}}+1$.
\end{lem}

\begin{proof}
It suffices to prove the result for minor-minimal graphs $G$ with average degree at least $d$. By \cref{BasicMinorMinimal}, each edge of $G$ is in at least $\floor{\frac{d}{2}}$ triangles. Say $G$ has $n$ vertices and $m$ edges. Thus $m\geq\ceil{\frac{dn}{2}}$. If $m>\ceil{\frac{dn}{2}}$ then deleting any one edge maintains the average degree condition, thus contradicting the minimality of $G$. Hence $m=\ceil{\frac{dn}{2}}$, and  $G$ has average degree 
$\frac{2m}{n}=\frac{2}{n}\ceil{\frac{dn}{2}} <
\frac{2}{n} (\frac{dn}{2}+1)
= d + \frac{2}{n}
\leq \frac{d^2+d+2}{d+1}$ since $n\geq d+1$. Thus $G$ has a vertex $v$ with degree at most 
$\ceil{\frac{d^2+d+2}{d+1}-1}=\ceil{\frac{d^2+1}{d+1}}$. Hence the subgraph of $G$ induced by the neighbours of $v$ has at most $\ceil{\frac{d^2+1}{d+1}}$ vertices and minimum degree at least $\floor{\frac{d}{2}}$. Moreover, the subgraph of $G$ induced by the closed neighbourhood of $v$ has at most $\ceil{\frac{d^2+1}{d+1}}+1$ vertices and minimum degree at least $\floor{\frac{d}{2}}+1$. 
\end{proof}

\citet{Mader68} introduced the following key definition. For an integer $k\geq 1$, let $X_k$ be the set of graphs $G$ with $|V(G)|\geq k$ and $|E(G)|\geq k|V(G)|-\tbinom{k+1}{2}$. 

\begin{lem}
\label{MinorMinimal}
Let $G$ be a  minor-minimal graph in $X_k$. Then $|E(G)|=k|V(G)|-\tbinom{k+1}{2}$, and either $G$ is isomorphic to $K_k$ or the neighbourhood of each vertex in $G$ induces a subgraph with minimum degree at least $k$. 
\end{lem}

\begin{proof}
Say $G$ has $n$ vertices and $m$ edges. Then $m=k|V(G)|-\tbinom{k+1}{2}$, otherwise delete an edge. If $n=k$ then $m= k^2-\binom{k+1}{2}=\binom{k}{2}$, implying that $G$ is isomorphic to $K_k$, as desired. Now assume that $n\geq k+1$. Let $vw$ be an edge of $G$. Say $vw$ is in $t$ triangles. Then $G/vw$ has $n-1\geq k$ vertices and $m-t-1$ edges. Since $G$ is minor-minimal, $G/vw$ is not in $X_k$. Thus $$kn-\tbinom{k+1}{2}-t-1= m-t-1= |E(G/vw)|\leq k(n-1)-\tbinom{k+1}{2}-1,$$ implying $t\geq k$. That is, each edge is in at least $k$ triangles. Therefore, the neighbourhood  of each vertex induces a subgraph with minimum degree at least $k$. 
\end{proof}

The following lemma is proved by mimicking a proof by \citet{Mader68} for the special case of $c_1=4$ and $c_2=\frac{3}{2}$. 

\begin{lem}
\label{NewMader} 
Fix constants $c_1>2$ and $c_2>1$. 
For each integer $k\geq 1$, every graph $G$ with average degree at least $4k$ has a minor with:\\
\textup{(1)} at most $(\frac{c_1}{2}+1)k$ vertices and minimum degree at least $2k$, or\\
\textup{(2)} at most $2k+1$ vertices and minimum degree at least $(1+\frac{1}{c_1})k$, or\\
\textup{(3)} at most $c_2k$ vertices and minimum degree at least $k$, or\\
\textup{(4)} at most $(4-\frac{c_1}{2})k$ vertices and minimum degree at least $c_2k$, or\\
\textup{(5)} $k$ vertices and minimum degree $k-1$ (that is, $K_k$). 
\end{lem}

\begin{proof}
Since $G$ has average degree at least $4k$, $G\in X_{2k}$. Let $G'$ be a minor-minimal minor of $G$ in $X_{2k}$. Thus $|E(G')|=2k|V(G')|-\binom{2k+1}{2}$. Hence $G'$ has average degree less than $4k$. Let $v$ be a vertex in $G'$ with degree less than $4k$. If $G'$ is isomorphic to $K_{2k}$ then outcome (5) holds. Otherwise, by \cref{MinorMinimal},  $G_0:=G'[N(v)]$ has minimum degree at least $2k$. If $G_0$ has at most $(\frac{c_1}{2}+1)k$ vertices then it satisfies outcome (1) and we are done. Now assume that $G_0$ has at least $(\frac{c_1}{2}+1)k$ vertices. Since $G_0$ has minimum degree at least $2k$, $|E(G_0)|\geq k|V(G_0)|$. Delete edges from $G_0$ until $|E(G_0)|=k|V(G_0)|$. 

Define $k':=\floor{\frac{c_1}{2}k}$. We now define a sequence of graphs $G_0,G_1,\dots,G_{k'}$ that satisfy  $|V(G_{i})|=|V(G_0)|-i$ and $$k|V(G_i)|-\frac{ik}{c_1}\leq |E(G_i)|\leq k|V(G_i)|.$$ Given $G_i$ where $0\leq i\leq k'-1$, construct $G_{i+1}$ as follows. First suppose that each edge $e$ in $G_i$ is in at least $(1+\frac{1}{c_1})k-1$ triangles. Since $|E(G_i)|\leq k|V(G_i)|$, some vertex $v$ in $G_i$ has degree at most $2k$. Thus the closed neighbourhood of $v$ induces a subgraph with at most $2k+1$ vertices and minimum degree at least $(1+\frac{1}{c_1})k$, which satisfies outcome (2). Now assume  some edge $e$ in $G_i$ is in at most $(1+\frac{1}{c_1})k-1$ triangles. Let $G_{i+1}$ be obtained from $G_i$ by contracting $e$. Thus $|V(G_{i+1})|=|V(G_i)|-1=|V(G_0)|-i-1$ and $$|E(G_{i+1})|\geq |E(G_i)|-(1+\frac{1}{c_1})k\geq k|V(G_i)|-\frac{ik}{c_1} -(1+\frac{1}{c_1})k =k|V(G_{i+1})|-\frac{(i+1)k}{c_1}.$$ If $G_{i+1}$ has more than $k|V(G_{i+1})|$ edges, then delete edges until $|E(G_{i+1})|=k|V(G_{i+1})|$. Thus $G_{i+1}$ satisfies the stated properties. 

Consider the final graph $F:=G_{k'}$. It satisfies 
\begin{align*}
|V(F)|&=|V(G_0)|-k'\geq (\frac{c_1}{2}+1)k-k'\geq k\text{ and }\\
|E(F)|&\geq k|V(F)|-\frac{k'}{c_1}k
\geq k|V(F)|-\tbinom{k+1}{2}.
\end{align*} 
Thus $F\in X_k$. Let $F'$ be a minor-minimal minor of $F$ in $X_k$. Thus $|V(F')|\leq|V(F)|=|V(G_0)|-k'\leq (4-\frac{c_1}{2})k$.
If $F'$ is isomorphic to $K_k$ then outcome (5) holds. Otherwise, by \cref{MinorMinimal}, the neighbourhood of each vertex in $F'$ induces a subgraph with minimum degree at least $k$. If $F'$ has a vertex of degree at most $c_2k$ then $F'[N(v)]$ has at most $c_2k$ vertices and has minimum degree at least $k$, which satisfies outcome (3). Otherwise $F'$ has minimum degree at least $c_2k$ and at most $(4-\frac{c_1}{2})k$ vertices, which satisfies outcome (4). 
\end{proof}

%
%

%
%
%
%

It is natural to maximise the ratio between the minimum degree and the number of vertices in our minor. The next lemma does that\footnote{The optimised constants used in the proof of \cref{Ratio} (and elsewhere in the paper) were computed using AMPL and the MINOS 5.5 nonlinear equation solver.}:

\begin{lem} 
\label{Ratio}
For every integer $k\geq 1$, every graph  with average degree at least $4k$ contains a complete graph $K_k$ as a minor  or contains a minor with $n$ vertices and minimum degree  $\delta$, where $\delta\geq 0.6518n$ and $2\delta-n\geq 0.4659k$ and $k\leq\delta < n\leq 4k$. 
\end{lem}

\begin{proof}
Apply \cref{NewMader} with $c_1=3.2929$ and $c_2=1.5341$. 
\end{proof} 

Maximising the difference between twice the minimum degree and the number of vertices in our minor will be useful below. The next lemma does that. 

\begin{lem} 
\label{Function}
For every integer $k\geq 1$, every graph with average degree at least $4k$ contains a complete graph $K_k$ as a minor  or contains a minor with $n$ vertices and minimum degree $\delta$, where $\delta\geq 0.6273n$  and $2\delta-n\geq 0.5773k$ and $k\leq \delta <n\leq 4k$. 
\end{lem}

\begin{proof}
Apply \cref{NewMader} with  $c_1=3.4641$ and $c_2=1.4227$. 
\end{proof}

\section{Deterministic Linear Bounds}
\label{sec:Linear}

This section establishes a number of linear bounds on $f(H)$. All the proofs are deterministic. 
The following well known lemma will be useful. We include the proof for completeness.

\begin{lem}
\label{TreeSubgraph}
Every graph $G$ with minimum degree at least $\ell-1$ contains every tree on $\ell\geq2$ vertices as a subgraph. 
\end{lem}

\begin{proof}
We proceed by induction on $\ell$ (with $G$ fixed). The base case with $\ell=1$ is trivial. Assume that $\ell\geq2$. Let $T$ be a tree on $\ell$ vertices. Let $v$ be a leaf of $T$ adjacent to  $w$. By induction, $G$ contains a subgraph $X$ isomorphic to $T-v$. Let $w'$ be the image of $w$ in $X$. Since $\deg_G(w')\geq\ell-1>|V(X-w')|=\ell-2$, there is a neighbour $v'$ of $w'$ in $G-X$. Mapping $v$ to $v'$ gives a subgraph of $G$ isomorphic to $T$. 
\end{proof}

A graph $G$ is \emph{2-degenerate} if every non-empty subgraph of $G$ has a vertex of degree at most 2. 

\begin{lem} 
\label{Degen}
Let $G$ be a graph with $n\geq 1$ vertices and minimum degree $\delta$, with $2\delta-n\geq t-2$. 
Then $G$ contains every 2-degenerate graph on $t\geq 1$ vertices as a subgraph.
\end{lem}

\begin{proof}
We proceed by induction on $t\geq 1$ (with $G$ fixed). The result is trivial for $t=1$.
Let $H$ be a 2-degenerate graph on $t$ vertices. 

First suppose that there is a degree-1 vertex $v$ in $H$ adjacent to $x$. 
By induction, $H-v$ is a subgraph of $G$.
Let $x'$ be the image of $x$ in $G$. Since $2\delta-n\geq t-2$ and $n>\delta$, we have 
 $\deg_G(x')\geq\delta> t-2$. Thus some neighbour of $x'$ is not used by the $t-2$ vertices in $H-x-v$. Embed $v$ at this neighbour, to obtain $H$ as a subgraph of $G$. 
 
Now assume that $H$ has minimum degree 2. Since $H$ is 2-degenerate, there is a degree-2 vertex $v$ in $H$  adjacent to $x$ and $y$. By induction, $H-v$ is a subgraph of $G$.
Let $x'$ and $y'$ be the images of $x$ and $y$ in $G$. Say $x'$ and $y'$ have $c$
common neighbours. Thus $x'$ has at least $\delta-c-1$ neighbours that are not $y'$ and not adjacent to $y'$.
Similarly, $y'$ has at least $\delta-c-1$ neighbours that are not $x'$ and not adjacent to $x'$. 
Thus $n\geq 2+c + 2(\delta-c-1)=2\delta-c$, implying $c \geq 2\delta-n\geq t-2$.
At most $t-3$ of the common neighbours of $x'$ and $y'$ are used by $H-v$.
So embed $v$ at one of the remaining common neighbours of $x'$ and $y'$.
And $H$ is a subgraph of $G$.
\end{proof}

\begin{lem}
\label{2degen}
Every graph $G$ with average degree at least $6.929t$ contains every 2-degenerate graph $H$ on $t\geq 1$ vertices as a minor.
\end{lem}

\begin{proof} 
If $t\leq 4$ then $G$ contains $K_4$ and thus $H$ as a minor (since $6.929t> 4=f(K_4)$). Now assume that $t\geq 5$. By assumption, $G$ has average degree at least $4k$, where $k:=\ceil{(t-2)/0.5773}$. If $G$ contains a $K_k$ minor, then  $G$ contains $H$ as a minor (since $t\geq 5$ implies $k\geq t$). Otherwise, by \cref{Function}, $G$ contains  a minor $G'$ with $n$ vertices and minimum degree $\delta$ where $k\leq\delta\leq n\leq 4k$ and $2\delta-n\geq 0.5773 k\geq t-2$. By \cref{Degen}, $G'$ contains $H$ as a subgraph. Thus $G$ contains $H$ as a minor. 
\end{proof}

We obtain the following straightforward linear bound for forcing an $H$-minor. The \emph{1-subdivision} of a graph $H$ is the graph obtained from $H$ by subdividing each edge of $H$ exactly once. A \emph{$(\leq 1)$-subdivision} of  $H$ is a graph obtained from $H$ by subdividing each edge of $H$ at most once. 

\begin{prop}
\label{BasicBasic}
Let $H$ be a graph with $t$ vertices and $q$ edges. Then every graph $G$ with average degree at least $6.929(t+q)$ contains $H$ as a minor. 
\end{prop}

\begin{proof}
If $H'$ is the 1-subdivision of $H$, then $H'$ has $t+q$ vertices and is 2-degenerate. By \cref{2degen}, $G$  contains $H'$ and thus $H$ as a minor. 
\end{proof}

This result is improved in \cref{New} below. First we need the following easy generalisation of \cref{Degen}. A subgraph $H'$ of a graph $H$ is \emph{spanning} if $V(H')=V(H)$. 

\begin{lem} 
\label{DegenSubgraph}
Let $H$ be a graph with $t\geq1$ vertices and $q$ edges. Assume that $H$ contains a 2-degenerate spanning subgraph $H'$ with $q'$ edges. Let $G$ be a graph with $n\geq 1$ vertices and minimum degree $\delta$, with $2\delta-n\geq q-q'+t-2$. Then $G$ contains a $(\leq 1)$-subdivision of $H$ as a subgraph.
\end{lem}

\begin{proof}
Let $H''$ be the graph obtained from $H$ by subdividing each edge not in $H'$ once. Thus $H''$ is 2-degenerate, and has $t+q-q'$ vertices. By \cref{Degen}, $G$ contains $H''$ as a subgraph, which is a  $(\leq 1)$-subdivision of $H$.
\end{proof}


\begin{lem} 
\label{BasicSubgraph}
Let $H$ be a graph with $t\geq1$ vertices and $q$ edges. Assume that $H$ contains a 2-degenerate spanning subgraph $H'$ with $q'$ edges. Let $G$ be a graph with $n\geq 1$ vertices and average degree at least $6.929(q-q'+t)$. Then $G$ contains $H$ as a minor. 
\end{lem}

\begin{proof} 
If $t\leq 4$ then $G$ contains $K_4$ and thus $H$ as a minor (since $6.929(q-q'+t)>4=f(K_4)$). Now assume that $t\geq 5$. By assumption, $G$ has average degree at least $4k$, where $k:=\ceil{(q-q'+t-2)/0.5773}$. If $G$ contains a $K_k$ minor, then $G$ contains $H$ as a minor (since $t\geq 5$ implies $k\geq t$). Otherwise, by \cref{Function}, $G$ contains  a minor $G'$ with $n$ vertices and minimum degree $\delta$ where $k\leq\delta\leq n\leq 4k$ and $2\delta-n\geq 0.5773 k\geq q-q'+t-2$. By \cref{DegenSubgraph}, $G'$ contains $H$ as a subgraph. Thus $G$ contains $H$ as a minor. 
\end{proof}

%
%
%


\begin{thm}
\label{New}
For every graph $H$ with $i$ isolated vertices and $q$ edges, 
every graph $G$ with average degree at least $i+6.929q$ contains $H$ as a minor. 
\end{thm}

\begin{proof}
Let $c:=6.929$. Let $t:=|V(H)|$. First note that $q\geq\frac{t-i}{2}$. 
We proceed by induction on $|V(H)|+|V(G)|$. The result is trivial if $|V(H)|\leq 1$. Now assume that $|V(H)|\geq 2$. Let $G$ be a graph with $n$ vertices, $m$ edges, and average degree $\frac{2m}{n}\geq i+cq$.
We may assume that $G$ is minor-minimal with average degree at least $i+c q$. By \cref{BasicMinorMinimal}, $G$ has minimum degree at least $\floor{\frac{i+cq}{2}}+1\geq q$. 

First suppose that  $H$ contains an isolated vertex $v$. Let $w$ be a vertex of minimum degree in $G$. Thus $\deg(w)\leq\frac{2m}{n}$. Hence the average degree of $G-w$ is
$$\frac{2(m-\deg(w))}{n-1}\geq
\frac{2m-\frac{2m}{n}-(n-1)}{n-1}=
\frac{2m}{n}-1\geq (t-1)+c q.$$ By induction, $G-w$ contains $H-v$ as a minor. Thus $G$ contains $H$ as a minor (with $v$ embedded at $w$). Now assume that $i=0$. 

Now suppose that some component $T$ of $H$ is a tree. Let $\ell:=|V(T)|$. Since $H$ has no isolated vertex,  $\ell\geq 2$. Also, $q=|E(H)|\geq|E(T)|=\ell-1$ and $G$ has minimum degree at least $\ell-1$. By \cref{TreeSubgraph}, there is a subgraph $T'$ of $G$ isomorphic to $T$. Let $G':=G-V(T')$. Note that $|E(G')|>m-\ell n$. 
By assumption, (a) $2m\geq cqn$. 
Since $c\geq4$ and $\ell\geq 2$, we have $c(\ell-1)\geq 2\ell$, implying (b) $-2\ell n \geq -c(\ell-1)n$. 
Also $q\geq\ell-1$, implying (c) $0 \geq - c \ell q + c \ell(\ell-1)$. 
Adding (a), (b) and (c) gives  
$$2m-2\ell n  \geq cqn - c \ell q -c(\ell-1)n + c \ell(\ell-1) =  c(q-(\ell-1))(n-\ell).$$
Hence the average degree of $G'$ is 
$$\frac{2|E(G')|}{|V(G')|}\geq \frac{2(m-\ell n)}{n-\ell}\geq c(q-(\ell-1)).$$

$H-V(T)$ has no isolated vertices and $q-(\ell-1)$ edges. By induction, $G'$ contains $H-V(T)$ as a minor. Hence $G$ contains $H$ as a minor, with $T$ mapped to $T'$. Now assume that no component of $H$ is a tree: Thus $q\geq t$. 

Let $H_1,\dots,H_k$ be the components of $H$. Each $H_i$ contains a spanning subgraph $H'_i$ consisting of a tree plus one edge. Let $H':=H'_1\cup\dots\cup H'_k$. Thus $|E(H'_i)|=|V(H_i)|$ and $|E(H')|=|V(H)|=t$. Observe that $H'$ is 2-degenerate. By \cref{BasicSubgraph} with $q'=t$, $G$ contains $H$ as a minor. 
\end{proof}

Note that the entire proof of \cref{New} is deterministic and leads to an algorithm for finding an $H$-minor in $G$ that has time complexity polynomial in both $|V(H)|$ and $|V(G)|$. 

\section{Probabilistic Linear Bounds}

This section applies the probabilistic method to improve the linear bounds in \cref{New}.

\begin{lem}
\label{Subgraph}
Let $H$ be a graph with $t$ vertices and $q$ edges. Let $G$ be a graph with $n\geq t$ vertices and average degree at least $d$. Then there is a spanning subgraph $R$ of $H$ with at least $\frac{dq}{n-1}$ edges, such that $R$ is isomorphic to a subgraph of $G$.
\end{lem}

\begin{proof}
Say $G$ has $m$ edges. Then $m\geq\half dn$. 
Let $f$ be a random injection $f$ from $V(H)$ to $V(G)$. Then by the linearity of expectation, 
\begin{align*}
\mathbb{E}(|\{vw\in E(H):f(v)f(w)\in E(G)\}|)
&= \sum_{vw\in E(H)} \mathbb{P}(f(v)f(w)\in E(G))\\
&= \sum_{vw\in E(H)} \frac{m}{\binom{n}{2}}\\
&\geq  \frac{dq}{n-1}.
\end{align*}
Thus there exists an injection $f$ from $V(H)$ to $V(G)$ such that 
$|\{vw\in E(H):f(v)f(w)\in E(G)\}|\geq \frac{dq}{n-1}$. 
Then the spanning subgraph $R$ of $H$ with $E(R):=\{vw\in E(H):f(v)f(w)\in E(G)\}$ satisfies the claim. 
\end{proof}

\begin{lem}
\label{1Sub}
Let $H$ be a graph with $t$ vertices and $q$ edges. Let $G$ be a graph with at most $n$ vertices and minimum degree at least $\delta$, such that  $$2\delta+4+\frac{\delta q}{n-1} \geq n+t+q.$$
Then $G$ contains a $(\leq 1)$-subdivision of $H$ as a subgraph. 
\end{lem}

\begin{proof}
By \cref{Subgraph},  there is a spanning subgraph $R$ of $H$ with at least $\frac{q\delta}{n-1}$ edges, such that $R$ is isomorphic to a subgraph of $G$. For each vertex $v$ of $H$, let $v'$ be the corresponding vertex of $G$ (defined by this isomorphism). Observe that the number of edges $vw$ of $H$ such that $v'w'$ is not an edge of $G$ is at most $q(1-\frac{\delta}{n-1})$. For each such edge we choose a common neighbour of $v'$ and $w$' and route $vw$ by a path in $G$ with one internal vertex. Consider each edge $vw$ of $H$ such that $v'w'$ is not an edge of $G$ in turn. Both $v'$ and $w'$ have degree at least $\delta$ and they are not adjacent. Thus $v'$ and $w'$ have at least $2\delta-(n-2)$ common neighbours. Since  $2\delta-(n-2)\geq (t-2)+q(1-\frac{\delta}{n-1})$, there is a common neighbour $x$ of $v'$ and $w'$ that is not already used by a vertex in $V(H)\setminus\{v,w\}$ or by a division vertex already assigned. Hence we may route $vw$ by the path $v'xw'$ in $G$. 
\end{proof}

Now we combine \cref{NewMader} and \cref{1Sub}. 

%
%
%
%

\begin{lem}
\label{NewBasicMethod}
Let $c_1>2$ and $c_2>1$. Define $a_1:=\frac{c_1}{2}+1$, $b_1:=2$,   $a_2:=2$, $b_2:=1+\frac{1}{c_1}$,  $a_3:= c_2$, $b_3:=1$,  $a_4:=4-\frac{c_1}{2}$  and  $b_4:=c_2$. Assume that for $1\leq i\leq 4$, 
$$0<2b_i-a_i\leq 3\text{ and }b_i<a_i.$$
Let $\alpha\geq 4$ and $\beta$ be numbers such that for $1\leq i\leq 4$, 
$$\alpha\geq\frac{4}{2b_i-a_i}\text{ and }\beta\geq \frac{4(a_i-b_i)}{a_i(2b_i-a_i)}.$$ 
Then, for every  graph $H$ with $t$ vertices and $q$ edges,
$$f(H) \leq \alpha t+\beta q.$$
\end{lem}

\begin{proof}
We are given a $t$-vertex $q$-edge graph $H$ and  a graph $G$ with average degree at least $\alpha t+\beta q\geq 4k$, where $k:=\floor{\frac{1}{4}(\alpha t+ \beta q)}$.  Since  $c_1>2$ and $c_2>1$, \cref{NewMader} is applicable to $G$. If case (5) occurs in \cref{NewMader}, then $K_k$ is a minor of $G$, which implies that $H$ is a minor of $G$ (since $\alpha\geq4$ implies $t\leq k$). Now assume that case (i) occurs in \cref{NewMader} for some $i\in\{1,2,3,4\}$. Let $a:=a_i$ and $b:=b_i$. Thus $G$ contains a minor $G'$ with $n\leq ak+1$ vertices and minimum degree $\delta\geq bk$. By the assumptions, 
\begin{align*}
(2b-a)k+3 \,\geq\, (2b-a)(k+1) \,>\, \frac{(2b-a)(\alpha t+\beta q)}{4} 
\,\geq\, t+\brac{\frac{a-b}{a}}q \,=\, t+\brac{1-\frac{b}{a}}q.
\end{align*}
Thus
\begin{align*}
2bk+4 +\frac{b}{a}\,q \,\geq\, t+q+ak+1.
\end{align*}
Since $n\leq ak+1$ and $\delta\geq bk$, we have $\frac{b}{a}\leq \frac{\delta}{n-1}$. Thus
\begin{align*}
2\delta+4+ \frac{\delta q}{n-1} \,\geq\, t+q+n.
\end{align*}
By \cref{1Sub}, $G'$ contains a $(\leq 1)$-subdivision of $H$ as a subgraph. 
Hence $G$ contains $H$ as a minor.
\end{proof}

Optimising $\beta$ in \cref{NewBasicMethod}  gives:

\begin{prop}
\label{qMain}
For every  graph $H$ with $t$ vertices and $q$ edges,
$$f(H) \leq 7.477t+2.375q.$$
\end{prop}

\begin{proof}
Apply \cref{NewBasicMethod} with $\alpha = 7.477$ and $\beta = 2.375$ and $c_1= 3.375$ and $c_2 = 1.465$. 
\end{proof}

%
%
%
%

We now prove the  bound introduced in Section~\ref{sec:Intro}.

\begin{proof}[Proof of \cref{pMain}] 
We proceed by induction on $t$ with the following hypothesis: Every graph $G$ with average degree at least $t+cq$ contains every graph $H$ on $t$ vertices and $q$ edges as a minor, where  $c:=6.291$. The result is trivial if $t\leq 1$. 
Now assume that $t\geq 2$.
Let $G$ be a graph with $n$ vertices, $m$ edges, and average degree $\frac{2m}{n}\geq t+cq$. 
We may assume that $G$ is minor-minimal with average degree at least $t+c q$. By \cref{BasicMinorMinimal}, $G$ has minimum degree at least $\floor{\frac{t+cq}{2}}+1\geq q$. 

Case 1.  $H$ contains an isolated vertex $v$: Let $w$ be a vertex of minimum degree in $G$. Thus $\deg(w)\leq\frac{2m}{n}$. Hence the average degree of $G-w$ is
$$\frac{2(m-\deg(w))}{n-1}\geq
\frac{2m-\frac{2m}{n}-(n-1)}{n-1}=
\frac{2m}{n}-1\geq (t-1)+c q.$$ By induction, $G-w$ contains $H-v$ as a minor. Thus $G$ contains $H$ as a minor (with $v$ embedded at $w$). Now assume that $H$ has no isolated vertex. 

Case 2.  Some component $T$ of $H$ is a tree: Let $\ell:=|V(T)|$. Note that $t=|V(H)|\geq|V(T)|=\ell\geq 2$ and $q=|E(H)|\geq|E(T)|=\ell-1$, implying that $G$ has minimum degree at least $\ell-1$. By \cref{TreeSubgraph}, there is a subgraph $T'$ of $G$ isomorphic to $T$. Let $G':=G-V(T')$. Note that $|E(G')|>m-\ell n$. By assumption, (a) $2m\geq (t+cq)n$. 
Since $c\geq2$ and $\ell\geq 2$, we have $c(\ell-1)\geq\ell$, implying (b) $-2\ell n \geq -(\ell+c\ell-c)n$. 
Since $q\geq\ell-1$,  we have (c) $0 \geq - c \ell (q -\ell+1)$. 
Since $t\geq1$,  we have  (d) $0 \geq  \ell - \ell t$. 
Adding (a), (b), (c) and (d) gives  
\begin{align*}
2m-2\ell n  \geq\; 
& (t+cq)n  -(\ell+c\ell-c)n   - c \ell (q -\ell+1) + \ell - \ell t \\
=\;&   n \big( (t-\ell) + c(q-\ell+1) \big)
-\ell \big( (t-\ell) + c(q-\ell+1) \big)\\
=\;&  \big( (t-\ell) + c(q-\ell+1) \big)(n-\ell).\end{align*}
Hence the average degree of $G'$ is 
$$\frac{2|E(G')|}{|V(G')|}\geq \frac{2(m-\ell n)}{n-\ell}\geq (t-\ell) + c(q-(\ell-1)).$$

Since $H-V(T)$ has $t-\ell$  vertices and $q-(\ell-1)$ edges, by induction, $G'$ contains $H-V(T)$ as a minor. Hence $G$ contains $H$ as a minor, with $T$ mapped to $T'$. Now assume that no component of $H$ is a tree. Thus $q\geq t$. 

Case 3. $t+cq\geq \alpha t + \beta q$, where $\alpha:=6.9687$ and $\beta:=2.484$: Then  $G$ has average degree at least $\alpha t + \beta q$, and thus contains $H$ as a minor by \cref{NewBasicMethod} with $c_1=3.484$ and $c_2=1.426$. 

Case 4: Now assume that $(\alpha-1)  t \geq (c-\beta)q$. Let $H_1,\dots,H_k$ be the components of $H$. Each $H_i$ contains a spanning subgraph $H'_i$ consisting of a tree plus one edge. Let $H':=H'_1\cup\dots\cup H'_k$. Thus $|E(H'_i)|=|V(H_i)|$ and $|E(H')|=|V(H)|=t$. Observe that $Q$ is 2-degenerate. 

Define $k:=\ceil{\frac{1}{4}(t+c(q-2))}$. Thus $G$ has average degree at least $t+cq>4k$. 
By \cref{Function}, $G$ contains a complete graph $K_k$ as a minor, or $G$ contains a minor $G'$ with $n'$ vertices and minimum degree $\delta$, where $2\delta-n'\geq \sigma k$ and $\sigma=0.5773$. In the first case, $H$ is a subgraph of $K_k$ (since $k\geq q\geq t$), implying $H$ is a minor of $G$. In the second case, 
\begin{align*}
2\delta-n'
\,\geq\,  \sigma k
\,\geq\, \frac{\sigma}{4}(t+c(q-2))
\,\geq\,   \frac{\sigma(c-\beta)}{4(\alpha-1)}q  +  \frac{\sigma c}{4}(q-2) 
\,\geq\,   q-2 ,
 \end{align*} 
 where the final inequality follows by considering the actual numerical values.  Thus, by \cref{DegenSubgraph} with $q'=t$,  $G'$ contains $H$ as a minor. Therefore $G$ contains $H$ as a minor. 
\end{proof}

The bound in \cref{pMain} is stronger than the bound in \cref{New} when $q\geq 1.567(t-i)$ (which is roughly when the non-isolated vertices in $H$ have average degree at least 3).

\section{General Result}
\label{sec:General}

The following lemma is at the heart of the proof of our main result (\cref{General}).

\begin{lem}
\label{heart}
For all $\lambda\in(\half,1)$ and $\epsilon\in(0,\lambda)$ there exists $d_0$ such that for every graph $H$ with $t$ vertices and average degree $d\geq d_0$, every graph $G$ with $n\geq(1+\epsilon)\ceil{\sqrt{\log_bd}}\, t$ vertices and minimum degree at least $\lambda n$ contains $H$ as a minor, where $b=(1-\lambda+\epsilon)^{-1}$.
\end{lem}

We first sketch the proof.  Say $V(H)=\{1,2,\dots,t\}$. Our goal is to exhibit disjoint subsets $X_1,\dots,X_t$ of $V(G)$ such that:
\begin{enumerate}[(a)] 
\item $G[X_i]$ is connected for $1 \leq i\leq t$, and
\item for each edge $ij$ of $H$ there is an edge of $G$ between $X_i$ and $X_j$. 
\end{enumerate}

We choose the $X_i$ in three stages. In the first two stages, we choose disjoint sets $S_1,\dots,S_t$ and $T_1,\dots,T_t$  randomly, with  the $S_i$ non-empty, such that:
\begin{enumerate}[(i)] 
\item every pair of vertices of $G$ have many common neighbours not in $S_1\cup \dots \cup S_t \cup T_1 \cup \dots \cup T_t$, 
\item for a small number of edges $ij \in E(H)$, there is no edge between $S_i \cup T_i$ and $S_j \cup T_j$, and
\item the total number (summed over all $i$) of components in $G[S_i \cup T_i]$ is small.
\end{enumerate}
Having done so, it is straightforward to greedily  chooses disjoint sets $U_1,\dots,U_t$, where $|U_i|$ equals the number of components of $G[S_i \cup T_i]$ minus 1, plus the number of edges $ij$ of $H$ with $j>i$ such that there is no edge of $GÕ$ between $S_i \cup T_i$ and $S_j \cup T_j$, so that (a) and (b) hold for   $X_i =S_i \cup T_i \cup U_i$.    

It remains to choose the $S_i$ and $T_i$ so that (i), (ii), and (iii) are satisfied.  In the first stage we randomly choose  disjoint sets $S_1,\dots,S_t$ each with $\ell=\ceil{\sqrt{\log_b d}\,}$  vertices.   In the second stage, we randomly choose the $T_i$ and show that  (i), (ii) and (iii) hold with positive probability. Some of the $T_i$  are empty, the rest of which have $2\ell^2$ vertices. $T_i$ is non-empty precisely if the size of the neighbourhood of $S_i$ is below a certain threshold. We need to  add the $T_i$ to such $S_i$ in the second phase to ensure that (ii) holds. In the first phase, we focus on bounding the number of $i$ for which the neighbourhood of $S_i$ is small. This allows us to bound the number of vertices used in the second phase, which helps in proving (i).  In the following proof, no effort is made to minimise $d_0$. 

\begin{proof}[Proof of \cref{heart}] 
Note that in $G$, every pair of vertices have at least $(2\lambda-1)n$ common neighbours (and $2\lambda-1>0$). Note that $b>1$. Let $\ell :=\ceil{\sqrt{\log_b d}\,}$. Define
\begin{align*}
\bad  := \brac{\frac{1-\lambda}{1-\lambda+\epsilon}}^\ell
\quad\text{and}\quad
\dis :=(1.5692)^{\ell}(1-\lambda)^{5\ell/6}.
\end{align*}
Observe that $0<\bad,\dis<1$ (since $\lambda>\half$), and $\bad$ and $\dis$ tend to 0 exponentially 
as $\ell\rightarrow\infty$. Now define
\begin{align*}
\theta  := 5(\bad+\dis)(\ell+\ell^2)+5\bad\ell^2+8.
\end{align*}
Elementary calculus shows that  $\theta$ is bounded by a function of $\epsilon$ and $\lambda$ independent of $\ell$. Thus, taking $d_0$ at least some function of $\epsilon$ and $\lambda$, since $d\geq d_0$, we may assume that $d$, $\ell$, $t$ and $n$ are at least functions of $\epsilon$, $\lambda$ and $\theta$. In particular, we assume:
\begin{align}
\epsilon(1-\epsilon)(2\lambda-1)\ell&\geq 2\theta	\label{Assumption1}\\
\exp\!\brac{\frac{\epsilon^2(2\lambda-1)^2n}{8}}&\geq 10\tbinom{n}{2} \label{Assumption2}\\
\exp\!\brac{\frac{\epsilon^4 n}{ 2(1+\epsilon)^2}}& \geq 10n \label{Assumption3}.
\end{align}

For a set $S$ of vertices in $G$, a vertex $v$ of $G$ is a \emph{non-neighbour} of $S$ if $v$ is not in $S$ and $v$ is not adjacent to a vertex in $S$. 

Say $V(H)=\{1,2,\dots,t\}$. Let $S_1,\dots,S_t,T_1,\dots,T_t$ be pairwise disjoint subsets of $V(G)$. Say $S_i$ is \emph{bad} if $S_i$ has at least $(n-\ell)(1-\lambda+\epsilon)^\ell$ non-neighbours, otherwise $S_i$ is \emph{good}. Say $S_i$ is \emph{disjointed} if $G[S_i]$ has a connected component with at most $\frac{\ell}{6}$ vertices. An edge $ij\in E(H)$ is \emph{problematic} if $S_i$ or $S_j$ is good (or both), but there is no edge in $G$ between $S_i$ and $S_j$. An edge $ij\in E(H)$ is \emph{nasty} if $S_i$ and $S_j$ are both bad and there is no edge in $G$ between $S_i\cup T_i$ and $S_j\cup T_j$. Below we prove the following two claims. 

\begin{claim} There exists subsets $S_1,\dots,S_t$ of $V(G)$ satisfying the following properties:\\
\PP{0}  $S_1,\dots,S_t$ are pairwise disjoint, and $|S_i|=\ell$ for $1\leq i\leq t$.\\
\PP{1} At most $5\bad t$ of the $S_i$ are bad. \\
\PP{2} At most $5\dis t$ of the $S_i$ are disjointed. \\
\PP{3} At most $\tfrac{5}{2}t$ edges of $H$ are problematic.\\
\PP{4} For all vertices $v,w\in V(G)$, 
$$|(N(v)\cap N(w))\setminus(S_1\cup\dots\cup S_t)|\;\geq\; \brac{1-\frac{\ell t}{n}-\frac{\epsilon}{2}}|N(v)\cap N(w)|.$$\\
\PP{5} For each vertex $v\in V(G)$,
$$|N(v)\setminus(S_1\cup\dots\cup S_t)|\;\geq\; \brac{1-\frac{\ell t}{n}-\frac{\epsilon^2}{\lambda(1+\epsilon)}}|N(v)|.$$
\end{claim}

\smallskip
\begin{claim} Given subsets $S_1,\dots,S_t$ of $V(G)$ that satisfy \PP{0}, \PP{1}, \PP{2}, \PP{3}, \PP{4} and \PP{5}, there exist subsets $T_1,\dots,T_t$ of $V(G)$ satisfying the following properties:

\smallskip
\QQ{0} $S_1,\dots,S_t,T_1,\dots,T_t$ are pairwise disjoint, and for $1\leq i\leq t$,
$$|T_i|=\begin{cases}
\ell^2 & \text{if $S_i$ is bad},\\ 
0 & \text{if $S_i$ is good}.
\end{cases}$$
\QQ{1}  At most $\frac{t}{2}$ edges of $H$ are nasty.
\end{claim}

Before proving these claims we show that they imply the lemma. By  \PP{0},
\begin{equation}
\label{SizeS}
|S_1\cup\dots\cup S_t| =\ell t,
\end{equation}
 and  by \PP{1} and \QQ{0},
\begin{equation}
\label{SizeT}
|T_1\cup\dots\cup T_t| \leq 5\bad t \cdot \ell^2.  
\end{equation} 
Mark each vertex in $\bigcup_i S_i\cup T_i$ as \emph{used}. 

For $i=1,2,\dots,t$, choose a set $U_i$ of less than $r_i$ vertices in $G$ as follows, where $r_i$ is the number of components of $G[S_i\cup T_i]$. Note that if $S_i$ is good and not disjointed, then $|S_i\cup T_i|=\ell$ and each component of $G[S_i\cup T_i]$ has more than $\frac{\ell}{6}$ vertices, implying $r_i\leq 5$. Otherwise (if $S_i$ is bad or disjointed) all we need is that $r_i\leq |S_i|+|T_i|\leq\ell+\ell^2$. For $1\leq j\leq r_i$, let $x_j$ be an arbitrary vertex in the $j$-th component of $G[S_i\cup T_i]$. For $j=1,\dots,r_i-1$, choose an unused common neighbour $z$ of $x_j$ and $x_{j+1}$, add $z$ to $U_i$, and mark $z$ as used. 

To prove that such a vertex $z$ exists, we first estimate $|\bigcup_iU_i|$. By \PP{1} and \PP{2}, at most $5(\bad  + \dis) t$ of the $S_i$ are bad or disjointed. Each of these  contribute at most $\ell+\ell^2$ vertices to $\bigcup_i  U_i$. For each $S_i$ that is good and not disjointed, at most $5$ vertices are added to $\bigcup_i U_i$. 
In total, by \eqref{SizeT}, 
\begin{align}
|\bigcup_iU_i| &\leq 5(\bad + \dis)(\ell+\ell^2) t +5t\quad\text{and} \label{SizeU}\\
|\bigcup_iT_i\cup U_i| &\leq 5(\bad + \dis)(\ell+\ell^2) t +5t + 5\bad\ell^2 t = (\theta-3)t.\label{SizeUT}
 \end{align}
By \PP{4} and \eqref{Assumption1} and \eqref{SizeUT}, and since $n\geq (1+\epsilon)\ell t$, 
\begin{align}
 |(N(x_j)\cap N(x_{j+1}))\setminus\bigcup_i(S_i\cup T_i\cup U_i) |
\;\geq\;& \brac{1-\frac{\ell t}{n}-\frac{\epsilon}{2}}(2\lambda-1)n - (\theta-3)t\nonumber\\
\;\geq\;& \brac{1-\frac{1}{1+\epsilon}-\frac{\epsilon}{2}}(2\lambda-1)(1+\epsilon)\ell t - (\theta-3)t\nonumber\\
=\;& \tfrac{\epsilon}{2}(1-\epsilon)(2\lambda-1)\ell t -(\theta-3)t\nonumber\\
\geq\;& 3t\;>\;0\nonumber
.
\end{align}
The used vertices are precisely $\bigcup_i(S_i\cup T_i\cup U_i)$. Thus the above inequality says that there is an unused common neighbour $z$ of $x_j$ and $x_{j+1}$, as claimed. By construction, each subgraph $G[S_i\cup T_i\cup U_i]$ is connected. 

Suppose that there is no edge in $G$ between $S_i\cup T_i\cup U_i$ and $S_j\cup T_j\cup U_j$ for some edge  $ij\in E(H)$. If $S_i$ or $S_j$ is good, then  $ij$ is problematic, otherwise $ij$ is nasty. Thus, by \PP{3} and \QQ{1} there are at most $3t$ such edges. Choose an unused common neighbour $z$ of some vertex in $S_i\cup T_i\cup U_i$ and some vertex in $S_j\cup T_j\cup U_j$, add $z$ to $U_i$, and mark $z$ as used. This step increases $|\bigcup_i U_i|$ by at most $3t$, implying that $|\bigcup_iT_i\cup U_i| \leq  \theta t$ by \eqref{SizeUT}. By the argument above, such a vertex $z$ exists. Now $S_1,\dots,S_t,T_1,\dots,T_t,U_1,\dots,U_t$ are pairwise disjoint, $G[S_i\cup T_i\cup U_i]$ is connected for each $i$, and for each edge $ij\in E(H)$, there is an edge in $G$ between $S_i\cup T_i\cup U_i$ and $S_j\cup T_j\cup U_j$. Thus $G$ contains $H$ as a minor (by contracting each set $S_i\cup T_i\cup U_i$). It remains to prove Claims 1 and 2. 


\begin{proof}[Proof of Claim~1]
Choose $S_1,\dots,S_t\subseteq V(G)$ satisfying \PP{0} uniformly at random. Since $n > \ell t=|S_1\cup\dots\cup S_t|$, such  subsets exist. We now bound the probability that each of \PP{1}, \PP{2}, \PP{3}, \PP{4} and \PP{5} fail. 

\PP{1}: Consider a subset $S_i$ and a vertex $v$ in $G-S_i$. Since $v$ has degree at least $\lambda n$ in $G$, and since $S_i$ is chosen at random in $V(G)$, for each vertex $x\in S_i$, the probability that $v$ is not adjacent to $x$ is at most $1-\lambda$. Thus the probability that $v$ is a non-neighbour of $S_i$ is at most $\brac{1-\lambda}^\ell$. By the linearity of expectation,  the expected number of non-neighbours of $S_i$ is at most $(n-\ell)(1-\lambda)^\ell$. Recall that $S_i$ is bad if $S_i$ has at least $(n-\ell)(1-\lambda+\epsilon)^\ell$ non-neighbours. Markov's inequality implies that the probability that $S_i$ is bad is at most
$$(n-\ell)\brac{1-\lambda}^\ell / (n-\ell)(1-\lambda+\epsilon)^\ell=\bad.$$
Thus the expected number of bad $S_i$ is at most $\bad t$. Since \PP{1} fails if the number of bad $S_i$ is more than $5\bad t$, Markov's inequality implies that \PP{1} fails with probability less than $\frac{1}{5}$. 

\PP{2}: Consider a disjointed set $S_i$. The number of subsets of $S_i$ with at most $\frac{\ell}{6}$ vertices is $$\sum_{j=0}^{\floor{\ell/6}}\binom{\ell}{j}
\leq 
2^{h(1/6)\ell}
<
(1.5692)^{\ell},$$
where $h(x)=-x \log_2 x - (1-x) \log_2 (1-x)$ is the binary entropy function.  Let $v$ be a vertex in a component of $G[S_i]$ with at most $\frac{\ell}{6}$ vertices. Thus $v$ is not adjacent to the at least $\frac{5}{6}\ell$ vertices in the other components of $G[S_i]$. These other vertices were chosen randomly. Thus the probability that $S_i$ is disjointed is less than $(1.5692)^{\ell}(1-\lambda)^{5\ell/6}=\dis$, and the expected number of disjointed $S_i$ is at most $\dis t$. Since \PP{2} fails if the number of disjointed $S_i$ is more than $5\dis t$, Markov's inequality implies that \PP{2} fails with probability less than $\frac{1}{5}$.

\PP{3}: Consider a  problematic edge $ij$ in $H$, where $S_i$ is good. Thus $S_i$ has at most $(n-\ell)(1-\lambda+\epsilon)^\ell$ non-neighbours. Since there is no edge between $S_i$ and $S_j$, every vertex in $S_j$ is one of these at most $(n-\ell)(1-\lambda+\epsilon)^\ell$ non-neighbours of $S_i$. Since $S_j$ is chosen randomly out of the $n-\ell$ vertices in $G-S_i$, the probability that each of the $\ell$ vertices in $S_j$ is a non-neighbour of $S_i$ is at most $(1-\lambda+\epsilon)^{\ell^2}\leq\frac{1}{d}$. (This is the key inequality in the whole proof.)\ Thus the probability that $ij\in E(H)$ is problematic is at most $\frac{1}{d}$. By the linearity of expectation,  the expected number of problematic edges (out of a total of $\frac{dt}{2}$) is at most $\frac{t}{2}$. Since \PP{3} fails if  the number of problematic edges is more than $\frac52t$, Markov's inequality implies that the probability that \PP{3} fails is less than $\frac{1}{5}$.

\PP{4}: Consider a pair of distinct vertices $v,w\in V(G)$. Let $X$ be the random variable $|(N(v)\cap N(w))\setminus(S_1\cup\dots\cup S_t)|$.  Since $S_1\cup\dots\cup S_t$ consists of $\ell t$ vertices chosen randomly from the $n$ vertices in $G$, we have $\mathbb{E}(X)= (1-\tfrac{\ell t}{n})|N(v)\cap N(w)|$.  If \PP{4} fails for $v,w$ then $X-\mathbb{E}(X)< -\tfrac{\epsilon}{2}|N(v)\cap N(w)|$. Hence 
$$\mathbb{P}(\text{\PP{4} fails for $v,w$}) \;\leq\; 
\mathbb{P}(|X-\mathbb{E}(X)| \geq \tfrac{\epsilon}{2}|N(v)\cap N(w)|)
.$$
The selection of $S_1,\dots,S_t$ may be considered as $\ell t$  trials, each choosing a random vertex from the vertices not already chosen. Changing the outcome of any one trial changes $\mathbb{E}(X)$ by at most 1.  Thus by Azuma's inequality\footnote{Azuma's inequality \citep{Azuma}  says that if $X$ is a random variable determined by $n$ trials $R_1,\dots,R_n$, such that for each $i$, and any two possible sequences of outcomes $r_1,\dots,r_i$ and $r_1,\dots,r_{i-1},r_i'$,
$$|\mathbb{E}(X\,|\,R_1=r_1,\dots,R_i=r_i)-\mathbb{E}(X\,|\,R_1=r_1,\dots,R_{i-1}=r_{i-1},R_i=r_i')|\leq c_i,$$
then $\mathbb{P}(|X-\mathbb{E}(X)|>x)\;\leq\;2\exp(-x^2/(2\sum_ic_i^2))$. In all our applications, $c_i=1$.} with 
$x=\tfrac{\epsilon}{2}|N(v)\cap N(w)|$,
\begin{align*}
\mathbb{P}(\text{\PP{4} fails for $v,w$}) 
\;\leq\;  \mathbb{P}(|X-\mathbb{E}(X)|>x)
\;\leq\;  2\exp\!\brac{\frac{-(\tfrac{\epsilon}{2}|N(v)\cap N(w)|)^2 }{2\ell t}}.
\end{align*} 
Since $|N(v)\cap N(w)|\geq(2\lambda-1)n$ and $n>\ell t$ and by \eqref{Assumption2},
\begin{align*}
\mathbb{P}(\text{\PP{4} fails for $v,w$}) 
\,\leq\,  2\exp\!\brac{\frac{-\epsilon^2(2\lambda-1)^2n^2}{8\ell t}}
\,<\,  2\exp\!\brac{\frac{-\epsilon^2(2\lambda-1)^2n}{8}}
\,\leq\,  \brac{5\binom{n}{2}}^{-1}.
\end{align*}
By the union bound, the probability that \PP{4} fails (for some pair of distinct vertices in $G$) is less than $\frac{1}{5}$. 

\PP{5}: Consider a vertex $v\in V(G)$. Let $X$ be the random variable $|N(v)\setminus(S_1\cup\dots\cup S_t)|$.  Since $S_1\cup\dots\cup S_t$ consists of $\ell t$ vertices chosen randomly from the $n$ vertices in $G$,  we have  $\mathbb{E}(X)= (1-\tfrac{\ell t}{n})|N(v)|$.  If \PP{5} fails for $v$ then 
$X-\mathbb{E}(X)< -\tfrac{\epsilon^2}{\lambda(1+\epsilon)}|N(v)|$. 
Hence 
$$\mathbb{P}(\text{\PP{5} fails for $v$}) 
\;\leq\;
\mathbb{P}(|X-\mathbb{E}(X)| \geq \frac{\epsilon^2}{\lambda(1+\epsilon)}|N(v)|).$$
As before, Azuma's inequality is applicable with $x=\tfrac{\epsilon^2}{\lambda(1+\epsilon)}|N(v)|$, giving
 $$\mathbb{P}(\text{\PP{5} fails for $v$}) 
 \;\leq\; 
 2\exp\!\brac{-\brac{\frac{\epsilon^2|N(v)|}{ \lambda(1+\epsilon) }}^2/\,2\ell t}    .$$
Since $|N(v)|\geq\lambda n$ and $n>\ell t$, and by  \eqref{Assumption3},
 $$\mathbb{P}(\text{\PP{5} fails for $v$}) 
 \;\leq\; 
 2\exp\!\brac{\frac{-\epsilon^4n^2}{ 2(1+\epsilon)^2\ell t}}
 \;<\; 2\exp\!\brac{\frac{-\epsilon^4n}{ 2(1+\epsilon)^2}}
 \;\leq\; (5n)^{-1}    .$$
By the union bound, the probability that \PP{5} fails (for some vertex in $G$) is less than $\frac{1}{5}$. 

We have shown that each of \PP{1}--\PP{5} fail with probability less than $\frac{1}{5}$. By the union bound, the probability that at least one of \PP{1}--\PP{5} fails is less than $1$. Thus the probability that none of \PP{1}--\PP{5} fails is greater than $0$. Thus there exists $S_1,\dots,S_t$ such that all of \PP{1}--\PP{5} hold.
\end{proof}

\begin{proof}[Proof of Claim~2]
Let $W:= V(G)\setminus(S_1\cup\dots\cup S_t)$. By \PP{5}, since $G$ has minimum degree at least $\lambda n$, and since $n\geq(1+\epsilon)\ell t$, the subgraph $G[W]$ has minimum degree at least 
\begin{align*}
\brac{1-\frac{\ell t}{n}-\frac{\epsilon^2}{\lambda(1+\epsilon)}}\lambda n
\;=\; 
\lambda(n-\ell t)-\frac{\epsilon^2n}{1+\epsilon}
\;\geq\; 
(\lambda-\epsilon)(n-\ell t)
\;=\; 
(\lambda-\epsilon)|W|.
\end{align*}
Choose $T_1,\dots,T_t\subseteq W$ satisfying \QQ{0} uniformly at random. Such subsets exist, 
since by \eqref{SizeT} and \eqref{Assumption1} (which implies that $\epsilon \geq 5\bad\ell$),
$$|W|\;=\; n-\ell t\;\geq\; \epsilon\ell t \;\geq\; 5\bad\ell^2 t \;\geq\; |T_1\cup\dots\cup T_t|.$$
If $ij\in E(H)$ is a nasty edge, and $v$ is any vertex in $T_j$,  then $v$ is adjacent to no vertex in $T_i$. Since $v$ and $T_i$ were chosen randomly in $W$, and $v$ is adjacent to at least $(\lambda-\epsilon)|W|$ vertices in $W$, the probability that $v$ is adjacent to no vertex in $T_i$ is at most 
$$\brac{1-(\lambda-\epsilon)}^{|T_i|} \;=\; b^{-\ell^2} \;\leq\;  b^{-\log_bd} \;=\; d^{-1}.$$ Thus the probability of an edge in $H$ being nasty is at most $d^{-1}$. Hence the expected number of nasty edges (out of a total of $\frac{dt}{2}$) is at most $\frac{t}{2}$. With positive probability the number of nasty edges is at most $\frac{t}{2}$. 
Hence there exists $T_1,\dots,T_t$ such that \QQ{1} holds. 
\end{proof}

This completes the proof of \cref{heart}. 
\end{proof}

\begin{proof}[Proof of \cref{General}]
Define $\epsilon:=0.00001$ and $\lambda:=0.6518$ and $b:=(1-\lambda+\epsilon)^{-1}>2.8718$. 
Let $H$ be a graph with $t$ vertices and average degree $d\geq d_0$, where $d_0$ is sufficiently large compared to $\epsilon$ and $\lambda$ (and thus an absolute constant). 
Let $G$ be a graph with average degree at least $3.895 \sqrt{\ln d}\,t$. 
Define $k:=\ceil{(1+\epsilon)\ceil{\sqrt{\log_bd}}\, t}$. 
Now $$3.895 \sqrt{ \ln b } >3.895 \sqrt{ \ln 2.8718 } > 4 (1+\epsilon).$$
Let $\eta:=3.895 - 4 (1+\epsilon)/ \sqrt{ \ln b }$, which is positive. 
Thus $3.895-\eta= 4(1+\epsilon)/\sqrt{\ln b}$ and
 $$(3.895-\eta)\sqrt{\ln d}\,t 
 =  4(1+\epsilon)\sqrt{\ln d}\,t /\sqrt{\ln b}
 = 4(1+\epsilon)\sqrt{\log_bd}\,t.$$ 
For sufficiently large $d_0$ and $d\geq d_0$, we have 
$\eta \sqrt{\ln d}\,t \geq4(1+\epsilon)t+4$. Adding these two inequalities gives
$$3.895 \sqrt{\ln d}\,t 
\geq
4(1+\epsilon)\sqrt{\log_bd}\,t \,+\,
4(1+\epsilon)t+4
\geq
4\ceil{(1+\epsilon)\ceil{\sqrt{\log_bd}}\, t}=4k. 
$$
Thus $G$ has average degree at least $4k$. 
By \cref{Ratio}, either $G$ contains $K_k$ as a minor or $G$ contains a minor $G'$ with $n>k$ vertices and minimum degree at least $\lambda n$. In the first case, $G$ contains $H$ as a minor (since $k\geq t$ for sufficiently large $d_0$ and $d\geq d_0$). In the second case, by \cref{heart}, there exists $d_0$ depending only on $\epsilon$ and $\lambda$, such that $G'$, and thus $G$, contains $H$ as a minor (assuming $d\geq d_0$).  
\end{proof}

\section{Open Problems}
\label{OpenProblems}

We conclude with a number of open problems that focus on $f(H)$ for various well-structured (non-random) graphs $H$. 

\begin{itemize}

\item Let $H$ consist of $k\geq1$ disjoint triangles. \citet{CorradiHajnal} proved that every graph of minimum degree at least $2k$ and order at least $3k$ contains $k$ disjoint cycles, and thus contains $H$ as a minor. Let $G$ be a graph with average degree at least $4k-2$ for some positive integer $k$.  By \cref{BasicMinorMinimal}, $G$ has a minor with minimum degree at least $2k$ and average degree at least $4k-2$ (implying the number of vertices is at least $4k-1\geq 3k$). By the above result of \citet{CorradiHajnal}, $G$ contains $H$ as a minor, and $f(H)\leq 4k-2$. (The same conclusion also follows from a result of \citet{Justesen}.)\ In fact, $f(H)=4k-2$ since if  $G$ is the complete bipartite graph $K_{2k-1, n}$ with $n\gg k$, then the average degree of $G$ tends to $4k -2$ as $n \rightarrow\infty$, but $G$ contains no $H$-minor since each cycle includes at least two vertices on each side.  
We conjecture the following generalisation: Every graph with average degree at least $\tfrac{4}{3}t-2$ contains every $2$-regular graph on $t$ vertices as a minor.

%


\item Fix integers $d\ll s\ll t$. Let $H_0$ be a $d$-regular graph on $t$ vertices. \citet{MT-Comb05} prove that $f(H_0)\geq c\sqrt{\log d}\,t$. Let $H$ be the graph obtained from $H_0$ by adding $s$ dominant vertices. Thus $H$ has average degree about $2s$. Hence  $c_1\sqrt{\log d}\,t\leq f(H_0)\leq f(H)\leq c_2\sqrt{\log s}\,t$ by \cref{General}. Where $f(H)$ lies between $c\sqrt{\log d}\,t$ and $c\sqrt{\log s}\,t$ is an interesting open problem. 

%

\item What is the least function $g$ such that every graph with average degree at least $g(k)\cdot t$ contains every graph with $t$ vertices and treewidth at most $k$ as a minor? Note that ``graph with $t$ vertices and treewidth at most $k$'' can be replaced by ``$k$-tree on $t$ vertices'' in the above. Since every such $k$-tree has less than $kt$ edges, Proposition~\ref{qMain} and \cref{General} respectively imply that $g(k)\leq 7.477+2.375 k$ and $g(k)\in \mathcal{O}(\sqrt{\log k})$. Since every 2-tree is 2-degenerate, $g(2)\leq 6.929$ by \cref{2degen}.

\item What is the minimum  constant $c$ such that every graph with average degree at least $ct^2$ contains the $t\times t$ grid as a minor? Since the $t\times t$ grid is 2-degenerate, $c\leq 6.929$ by \cref{2degen}.

\item What is the least constant $c$ such that every graph with average degree at least $c t$ contains every planar graph with $t$ vertices as a minor? Since such a planar graph has less than $3t$ edges, Proposition~\ref{qMain} implies that $c\leq 14.602$.

\item What is the least function $g$ such that every graph with average degree at least $g(k)\cdot t$ contains every $K_k$-minor-free graph with $t$ vertices as a minor? Since every $K_k$-minor-free graph has average degree  $\mathcal{O}(k\sqrt{\log k})$,  \cref{General} implies that $g(k)\in \mathcal{O}(\sqrt{\log k})$.

\item Every graph with average degree at least $10t^2$ contains a subdivision of $K_t$ as a subgraph. A proof of this result is given by \citet{Diestel4}  based on results on highly connected subgraphs by \citet{Mader72} and on linkages by \citet{TW05}. This method immediately generalises to prove that for every graph $H$ with $t$ vertices and $q$ edges, every graph with average degree at least $4t+20q$ contains a subdivision of $H$ as a subgraph. Determining the best constants in such a result is an interesting line of research. Note that there is a linear lower bound for a graph $H$ with $t$ vertices and $q$ edges, such that every set of at least $\frac{t}{2}$ vertices induces a subgraph with at least $\epsilon q$ edges, for some $\epsilon>0$. Say $K_{n,n}$ contains a subdivision of $H$. At least $\frac{t}{2}$ original vertices of $H$ are on one side of $K_{n,n}$. Thus at least $\epsilon q$ edges have a division vertex on the other side of $K_{n,n}$, implying $n\geq\epsilon q$. Hence, average degree at least $\epsilon q$ is needed to force a subdivision of $H$.

\end{itemize}

\subsection*{Acknowledgements}

This research was partially completed at a workshop held at the Bellairs Research Institute in Barbados in March 2013. Many thanks to all the participants for helpful discussions and creating a stimulating working atmosphere. Thanks to the referee for many helpful observations. 


\def\soft#1{\leavevmode\setbox0=\hbox{h}\dimen7=\ht0\advance \dimen7
  by-1ex\relax\if t#1\relax\rlap{\raise.6\dimen7
  \hbox{\kern.3ex\char'47}}#1\relax\else\if T#1\relax
  \rlap{\raise.5\dimen7\hbox{\kern1.3ex\char'47}}#1\relax \else\if
  d#1\relax\rlap{\raise.5\dimen7\hbox{\kern.9ex \char'47}}#1\relax\else\if
  D#1\relax\rlap{\raise.5\dimen7 \hbox{\kern1.4ex\char'47}}#1\relax\else\if
  l#1\relax \rlap{\raise.5\dimen7\hbox{\kern.4ex\char'47}}#1\relax \else\if
  L#1\relax\rlap{\raise.5\dimen7\hbox{\kern.7ex
  \char'47}}#1\relax\else\message{accent \string\soft \space #1 not
  defined!}#1\relax\fi\fi\fi\fi\fi\fi}

\end{document}